\documentclass[12pt]{amsart}
\usepackage{amsfonts,amssymb,amscd,amsmath,enumerate,verbatim}
\usepackage[latin1]{inputenc}
\usepackage{amscd}
\usepackage{latexsym}
\usepackage{pstcol,pst-plot,pst-3d}
\usepackage{mathptmx}
\usepackage{multicol}

\psset{unit=0.7cm,linewidth=0.8pt,arrowsize=2.5pt 4}

\newpsstyle{fatline}{linewidth=1.5pt}
\newpsstyle{fyp}{fillstyle=solid,fillcolor=verylight}
\definecolor{verylight}{gray}{0.97}
\definecolor{light}{gray}{0.9}
\definecolor{medium}{gray}{0.85}

\def\frk{\mathfrak}               

\def\Phi{{\frk N}}
\def\G{\Gamma}

\def\opn#1#2{\def#1{\operatorname{#2}}} 
\opn\chara{char} \opn\length{\ell} \opn\pd{pd} \opn\rk{rk}
\opn\projdim{proj\,dim} \opn\injdim{inj\,dim} \opn\rank{rank}
\opn\depth{depth} \opn\grade{grade} \opn\height{height} \opn\bheight{bigheight}
\opn\embdim{emb\,dim} \opn\codim{codim}

\opn\Tr{Tr} \opn\bigrank{big\,rank}
\opn\superheight{superheight}\opn\lcm{lcm}
\opn\trdeg{tr\,deg}
\opn\reg{reg} \opn\lreg{lreg} \opn\ini{in} \opn\lpd{lpd}
\opn\size{size}\opn{\mult}{mult}\opn{\lex}{lex}
\opn\div{div} \opn\Div{Div} \opn\cl{cl} \opn\Cl{Cl}
\opn\Spec{Spec} \opn\Supp{Supp} \opn\supp{supp} \opn\Sing{Sing}
\opn\Ass{Ass} \opn\Min{Min}
\opn\Ann{Ann} \opn\Rad{Rad} \opn\Soc{Soc}
\opn\Syz{Syz} \opn\Im{Im} \opn\Ker{Ker} \opn\Coker{Coker}
\opn\Am{Am} \opn\Hom{Hom} \opn\Tor{Tor} \opn\Ext{Ext}
\opn\End{End} \opn\Aut{Aut} \opn\id{id} \opn\ini{in}

\opn\nat{nat}
\opn\pff{pf}
\opn\Pf{Pf} \opn\GL{GL} \opn\SL{SL} \opn\mod{mod} \opn\ord{ord}
\opn\Gin{Gin}
\opn\Hilb{Hilb}\opn\adeg{adeg}\opn\std{std}\opn\ip{infpt}
\opn\Pol{Pol}
\opn\sat{sat}
\opn\Var{Var}
\opn\Gen{Gen}
\opn\indmatch{indmatch}

\opn\aff{aff} \opn\con{conv} \opn\relint{relint} \opn\st{st}
\opn\lk{lk} \opn\cn{cn} \opn\core{core} \opn\vol{vol}
\opn\link{link} \opn\star{star}
\opn\gr{gr}

\def\pot#1#2{#1[\kern-0.28ex[#2]\kern-0.28ex]}

\opn\dirlim{\underrightarrow{\lim}}
\opn\inivlim{\underleftarrow{\lim}}
\let\Union=\bigcup

\let\to=\rightarrow
\let\To=\longrightarrow
\def\Implies{\ifmmode\Longrightarrow \else
        \unskip${}\Longrightarrow{}$\ignorespaces\fi}
\def\implies{\ifmmode\Rightarrow \else
        \unskip${}\Rightarrow{}$\ignorespaces\fi}
\def\iff{\ifmmode\Longleftrightarrow \else
        \unskip${}\Longleftrightarrow{}$\ignorespaces\fi}

\let\:=\colon
\newtheorem{Theorem}{Theorem}[section]
\newtheorem{Lemma}[Theorem]{Lemma}
\newtheorem{Corollary}[Theorem]{Corollary}
\newtheorem{Proposition}[Theorem]{Proposition}
\newtheorem{Remark}[Theorem]{Remark}

\newtheorem{Definition}[Theorem]{Definition}

\let\epsilon\varepsilon
\let\phi=\varphi
\let\kappa=\varkappa
\def\qed{\ifhmode\textqed\fi
      \ifmmode\ifinner\quad\qedsymbol\else\dispqed\fi\fi}
\def\textqed{\unskip\nobreak\penalty50
       \hskip2em\hbox{}\nobreak\hfil\qedsymbol
       \parfillskip=0pt \finalhyphendemerits=0}
\def\dispqed{\rlap{\qquad\qedsymbol}}

\opn\dis{dis}
\def\pnt{{\raise0.5mm\hbox{\large\bf.}}}

\opn\Lex{Lex}

\newcommand{\inD}[1][\relax]{\def\argone{#1}\def\temprelax{\relax}
  \ifx\argone\temprelax\right.\else\,\middle|#1\right.{}\fi}

\newif\ifbinary
\binarytrue

\begin{document}

\title{Extremal Betti numbers of some classes of binomial edge ideals}

\author{Ahmet Dokuyucu}

\address{Faculty of Mathematics and Computer Science, Ovidius University\\
Bd. Mamaia 124, 900527 Constanta\\ and University of South-East Europe Lumina\\
Sos. Colentina nr. 64b, Bucharest \\
Romania} \email{ahmet.dokuyucu@lumina.org }

\begin{abstract}
Let $G$ be  a cycle or a complete bipartite graph. We show that the binomial edge ideal $J_{G}$ and its initial ideal  with respect to the lexicographic order have the same extremal Betti number. This is a partial 
positive answer to a conjecture proposed in \cite{EHH}.
\end{abstract}
\maketitle

\section*{Introduction}

Let $G$ be a  simple graph on the vertex set $[n]$ with edge set $E(G)$ and let $S$ be the polynomial ring $K[x_1,\ldots,x_n,y_1,\ldots,y_n]$ in $2n$ variables endowed 
with the lexicographic order induced by $x_1>\cdots >x_n>y_1>\cdots>y_n$. The binomial edge ideal $J_G\subset S$ associated with $G$ is generated by all the 
binomials $f_{ij}=x_iy_j-x_jy_i$ with $\{i,j\}\in E(G).$ The binomial edge ideals were introduced in \cite{HHHKR} and, independently, in \cite{Oh}. Meanwhile, many 
algebraic and homological properties of these ideals have been investigated; see, for instance, \cite{CR}, \cite{EHH}, \cite{EZ}, \cite{HHHKR}, \cite{MM}, \cite{RR}, \cite{Sara}, \cite{Sara2}, \cite{SZ}, \cite{Z}, \cite{ZZ}.

In \cite{EHH}, the authors conjectured that the extremal  Betti numbers of $J_G$ and $\ini_<(J_G)$ coincide for any graph $G.$ Here, $<$ denotes the lexicographic 
order in $S$ induced by the natural order of the variables. In this article, we give a positive answer to this conjecture when the graph $G$ is a complete bipartite graph or a cycle. To this aim, we use some results proved in \cite{SZ} and \cite{ZZ} 
which completely  characterize the resolution of the binomial edge ideal $J_G$ when $G$ is a cycle or a complete bipartite graph. In particular, in this case, it follows that $J_G$  has a unique extremal Betti number. In the first section we recall all the known facts on the resolutions of binomial edge ideals of the complete bipartite graphs and cycles. In   Section~\ref{main}, we study the initial ideal of $J_G$ when $G$ is a bipartite graph or a  cycle. We show that $\projdim \ini_<(J_G)=\projdim J_G$ and $\reg \ini_<(J_G)=\reg J_G,$ and,  therefore, $\ini_<(J_G)$ has a unique extremal Betti number as well. Finally, we show that the extremal Betti number of $\ini_<(J_G)$ is equal to that of $J_G.$

To our knowledge, this is the first attempt to prove the conjecture stated in \cite{EHH} for extremal Betti numbers. In our 
study, we take advantage of the known results on the resolutions of  binomial edge ideals of cycles and complete bipartite 
graphs and of the fact that their initial ideals have  nice properties. For instance, as we show in Section~\ref{main}, the 
initial ideal of $J_G$ for a complete bipartite graph  has linear quotients and is generated in degrees $2$ 
and $3.$ Therefore, it is componentwise linear and its Betti numbers may be computed easily (Theorem~\ref{bnbers}). The 
initial ideal of $J_G$ when $G$ is a cycle does not have  linear quotients, but by ordering its generators in a suitable 
way, we may easily compute its extremal Betti number (Theorem~\ref{final}). It is interesting to remark that even if the 
admissible paths of the cycle (in the sense of \cite[Section 3]{MM}) determine the minimal set of monomial generators of $
\ini_<(J_G),$ 
the Lyubeznik resolution \cite{L} does not provide a minimal resolution of $\ini_<(G).$

\section{Preliminaries}

\subsection{Binomial edge ideals of complete bipartite graphs} Let $G=K_{m,n}$ be the complete bipartite graph on the vertex 
set $\{1,\ldots,m\}\cup \{m+1,\ldots,m+n\}$ with $m\geq n\geq 1$ and let $J_G$ be its binomial edge ideal. $J_G$ is 
generated by all the binomials $f_{ij}=x_iy_j-x_jy_i$ where $1\leq i\leq m$ and $m+1\leq j\leq m+n.$ In \cite[Theorem 5.3]{SZ} it is shown that the Betti diagram of $S/J_G$ has the form
\[
\begin{array}{c|ccccc}
 & 0 & 1 & 2 & \cdots & p\\
 \hline
0 & 1 & 0 & 0 & \cdots & 0\\
1 & 0 & mn & 0 & \cdots & 0\\
2 & 0 & 0 & \beta_{2 4} & \cdots & \beta_{p, p+2}
\end{array}
\]
where $p=\projdim S/J_G=\left\{
\begin{array}{ll}
	m, & \text{ if } n=1,\\
	2m+n-2, & \text{ if } n>1.
\end{array}\right.
$

In particular, from the above Betti diagram we may read that $S/J_G$ has a unique extremal Betti number, namely $\beta_{p, p+2}.$

Moreover, in \cite[Theorem 5.4]{SZ} all the Betti numbers of $S/J_G$ are computed. Since we are interested only in the extremal Betti number, we recall here its value as it was given in \cite[Theorem 5.4]{SZ}, namely, $\beta_{p, p+2}=
\left\{
\begin{array}{ll}
	m-1, & \text{ if } p=m,\\
	n-1, & \text{ if } p=2m+n-2.
\end{array}\right.$

Since we will study the initial ideal of $J_G$ with respect to the lexicographic order induced by the natural order of the variables, we need to recall the following definition and result of \cite{HHHKR}.

\begin{Definition}{\em
\label{admissiblepath} Let $i<j$ be two vertices of an arbitrary graph $G.$
 A path $i=i_0,i_1,\ldots,i_{r-1},i_r=j$ from $i$ to $j$
is called \textbf{admissible} \index{path! admissible} if the following conditions are fulfilled:
\begin{enumerate}
	\item [{\em (i)}] $i_k\neq i_\ell$ for $k\neq\ell;$
	\item [{\em (ii)}] for each $k=1,\ldots,r-1,$ one has either $i_k < i$ or $i_k > j;$
	\item [{\em (iii)}] for any proper subset $\{j_1,\ldots,j_s\}$ of $\{i_1,\ldots,i_{r-1}\}$, the sequence $i,j_1,\ldots,j_s,j$ is not a path in $G.$
\end{enumerate}}
\end{Definition}

Given an admissible path $\pi$ of $G$ from $i$ to $j,$ we set $u_\pi=(\prod_{i_k >j}x_{i_k})(\prod_{i_\ell < i}y_{i_\ell}).$

Obviously, any edge of $G$ is an admissible path. In this case, the associated monomial is just $1.$

\begin{Theorem}[HHHKR]
\label{binomGB} Let $G$ be an arbitrary graph.
The set of binomials
\[
\G=\Union_{i<j}\{u_\pi f_{ij}: \pi \text{ is an admissible path from }i \text{ to }j\}
\]
is the reduced Gr\"obner basis of $J_G$ with respect to lexicographic order on $S$ induced by the natural order of indeterminates, $x_1>\cdots >x_n>y_1>\cdots >y_n$.
\end{Theorem}

One may easily see that the only admissible paths of the complete graph $G=K_{m,n}$ are the edges of $G$, the paths of the form $i,m+k,j$ with $1\leq i<j\leq m$, $1\leq k\leq m$, and $m+i, k,m+j$ with $1\leq i<j \leq n$, $1\leq k \leq m.$ Therefore, we get the following consequence of the above theorem.

\begin{Corollary}\label{inicomplete}
Let $G=K_{m,n}$ be the complete bipartite graph on the vertex set $V(G)=\{1,\ldots,m\}\cup\{m+1,\ldots,m+n\}.$ Then
\[
\ini_<(J_G)=(\{x_iy_j\}_{\stackrel{1\leq i\leq m} {m+1\leq j\leq m+n}}, \{x_ix_{m+k}y_j\}_{\stackrel{1\leq i<j\leq m}{1\leq k\leq n}}, \{x_{m+1}y_ky_{m+j}\}_{\stackrel{1\leq i<j \leq n} {1\leq k \leq m}}).
\]
\end{Corollary}

\subsection{Binomial edge ideals of cycles}
In this subsection, $G$ denotes the $n$--cycle on the vertex set $[n]$ with edges $\{1,2\},\{2,3\},\ldots, \{n-1,n\},\{1,n\}$.

In \cite{ZZ} it was shown that the Betti diagram of $S/J_G$ has the form

\[
\begin{array}{c|cccccc}
 & 0 & 1 & 2 & 3& \cdots & n\\
 \hline
0 & 1 & 0 & 0 & 0 &  \cdots & 0\\
1 & 0 & n & 0 & 0 &\cdots & 0\\
2 & 0 & 0 & \beta_{2 4} & 0 &\cdots & 0\\
3 & 0 & 0 & 0 & \beta_{3 6} & \cdots & 0\\
\vdots & \vdots & \vdots & \vdots & \vdots & \vdots &\vdots\\
n-2& 0 & 0 & \beta_{2,n} & \beta_{3, n+1} & \cdots & \beta_{n, 2n-2}
\end{array}
\]
and all the Betti numbers were computed. One sees that we have a unique extremal Betti number and, by \cite{ZZ}, we have $\beta_{n,2n-2}={n-1\choose 2}-1.$

We now look at the initial ideal of $J_G.$ It is obvious by Definition~\ref{admissiblepath} and by the labeling of the vertices of $G$ that the admissible 
paths are the edges of $G$ and the paths of the form $i, i-1,\ldots,1,n,n-1,\ldots,j+1$ with  $2\leq j-i\leq n-2.$ Consequently, we get the following system of generators for the initial ideal of $J_G.$

\begin{Corollary}\label{inicycle}
Let $G$ be the $n$--cycle with the natural labeling of its vertices. Then 
\[
\ini_<(J_G)=(x_1y_2,\ldots, x_{n-1}y_n,x_1y_n,\{x_ix_{j+1}\cdots x_ny_1\cdots y_{i-1}y_j\}_{2\leq j-i\leq n-2}).
\]
\end{Corollary}

\section{Extremal Betti numbers}
\label{main}

\subsection{Complete bipartite graphs} Let $G=K_{m,n}$ be the complete bipartite graph on  the vertex set $\{1,\ldots,m\}\cup \{m+1,\ldots,m+n\}$ with $m\geq n\geq 1$ and let $J_G$ be its binomial edge ideal. The initial ideal $\ini_<(J_G)$ has a nice property which is stated in the following proposition. 

\begin{Proposition}\label{linquot}
Let $G=K_{m,n}$ be the complete graph. Then $\ini_<(J_G)$ has linear quotients.
\end{Proposition}

\begin{proof}
Let $u_1,\ldots, u_r$ be the minimal generators of $\ini_<(J_G)$ given in Corollary~\ref{inicomplete} 
where, for $i<j$,  either $\deg u_i<\deg u_j$ or $\deg u_i=\deg u_j$ and $u_i>u_j.$ We show that we respect to this order of its minimal monomial generators, $\ini_<(J_G)$ has linear quotients, that is, for any $\ell>1,$ the 
ideal quotient  $(u_1,\ldots,u_{\ell-1}):(u_\ell)$ is generated by variables.

Let $u_\ell=x_iy_j$ for some $1\leq i\leq m$ and $m+1\leq j\leq m+n.$ In this case, one may easily check that 
\begin{equation}\label{eq1}
(u_1,\ldots,u_{\ell-1}):(u_\ell)=(x_1,\ldots,x_{i-1},y_{m+1},\ldots,y_{j-1}).
\end{equation}

Let $u_\ell=x_i x_{m+k}y_j$ for some $1\leq i<j\leq m$ and $1\leq k\leq n.$ Then we get 
\begin{equation}\label{eq2}
(u_1,\ldots,u_{\ell-1}):(u_\ell)=(y_{m+1},\ldots,y_{m+n},x_1,\ldots,x_{i-1},x_{m+1},\ldots,x_{m+k-1},y_{i+1},\ldots,y_{j-1}).
\end{equation}

Finally, if $u_\ell=x_{m+i} y_{k}y_{m+j}$ for some $1\leq i<j\leq n$ and $1\leq k\leq m,$ we have
\begin{equation}\label{eq3}
(u_1,\ldots,u_{\ell-1}):(u_\ell)=(x_1,\ldots,x_m,x_{m+1},\ldots,x_{m+i-1},y_1,\ldots,y_{k-1},y_{m+i+1},\ldots,y_{m+j-1}).
\end{equation}
\end{proof}

\begin{Theorem}\label{bnbers}
Let $G=K_{m,n}$ be the complete graph. Then
\[
\beta_{t,t+2}(\ini_<(J_G))=\sum_{\stackrel{1\leq i\leq m}{m+1\leq j\leq m+n}}{i+j-m-2\choose t},
\]
\[
\beta_{t,t+3}(\ini_<(J_G))=
\left\{
\begin{array}{ll}
\sum_{\stackrel{1\leq i<j\leq m}{1\leq k\leq n}}{n+k+j-3\choose t}, & \text{ if }	n=1,\\
\sum_{\stackrel{1\leq i<j\leq m}{1\leq k\leq n}}{n+k+j-3\choose t} +
\sum_{\stackrel{1\leq i<j\leq n}{1\leq k\leq m}}{m+k+j-3\choose t}, & \text{ if }	n>1.
\end{array}
\right.
\]
\end{Theorem}

\begin{proof}
Since $\ini_<(J_G)$ has linear quotients, we may apply \cite[Exercise 8.8]{HH10} and get
\[
\beta_{t,t+d}(\ini_<(J_G))=\sum_{\stackrel{\ell=1}{\deg u_{\ell}=d}}^m {q_\ell \choose t}
\]
where $q_\ell$ is the number of variables which generate $(u_1,\ldots,u_{\ell-1}):(u_\ell).$  Hence, by using equations (\ref{eq1})-(\ref{eq3}) for counting the number of variables which generate $(u_1,\ldots,u_{\ell-1}):(u_\ell)$, we get all the graded Betti numbers of $\ini_<(J_G)$.
\end{proof}

In particular, by the above theorem, it follows the following corollary which shows that for $G=K_{m,n}$ the extremal  Betti numbers of $S/J_G$ and $S/\ini_<(J_G)$ coincide.

\begin{Corollary}
Let $G=K_{m,n}$ be the complete graph. Then:

\begin{itemize}
	\item [(a)] $\projdim (S/\ini_<(J_G))=\projdim (\ini_<(J_G))+1=
	\left\{
	\begin{array}{ll}
		m,& \text{ if } n=1,\\
		2m+n-2,& \text{ if } n>1.
	\end{array}
	\right.$
	\item [(b)] $S/\ini_<(J_G)$ has a unique extremal Betti number, namely
	\[
	\beta_{p,p+2}(S/\ini_<(J_G)) = \beta_{p-1,p+2}(\ini_<(J_G))=\left\{
	\begin{array}{ll}
		m-1, & \text{ if } n=1,\\
		n-1, & \text{ if } n>1.
	\end{array}
	\right.
	\]
	\end{itemize}
\end{Corollary}

\begin{proof}
(a) follows immediately from Betti number formulas of Theorem~\ref{bnbers}. 

Let us prove (b). By using again Theorem~\ref{bnbers}, we get 
\[
\beta_{p-1,p+2}(\ini_<(J_G))=\left\{
	\begin{array}{ll}
		\sum_{1\leq i<m}{m-1\choose p-1}=\sum_{1\leq i<m}{m-1\choose m-1}=m-1, & \text{ if } n=1,\\
		\sum_{1\leq i<n}{2m+n-3\choose p-1}=\sum_{1\leq i<n}{2m+n-3\choose 2m+n-3}=n-1, & \text{ if } n>1.
	\end{array}
	\right.
\]
\end{proof}

\subsection{Cycles} In this subsection, the graph $G$ is an $n$--cycle. If $n=3,$ then $G$ is a complete graph, therefore 
the ideals $J_G$ and $\ini_<(J_G)$ have the same graded Betti numbers. Thus, in the sequel,  we may consider 
$n\geq 4.$ 

As we have already seen in Corollary~\ref{inicycle}, $\ini_<(J_G)$ is minimally generated by the initial monomials of the binomials corresponding to the edges of $G$  and by $m=n(n-3)/2$  monomials of degree $\geq 3$ which we denote by 
$v_1,\ldots, v_m$ where we assume that if $i<j$, then either $\deg v_i<\deg v_j$ or $\deg v_i=\deg v_j$ and $v_i>v_j$. Let us observe that if $v_k=x_ix_{j+1}\cdots x_ny_1\cdots y_{i-1}y_j,$ we have $\deg v_k=n-j+i+1.$ Hence, there are two monomials of degree $3,$ namely, $v_1=x_1x_ny_{n-1}$ and $v_2=x_2y_1y_n,$ three monomials of degree $4$, namely, $v_3=x_1x_{n-1}x_ny_{n-2}, v_4=x_1x_ny_1y_{n-1},v_5=x_1y_1y_2y_n$, etc.

We introduce the following notation. We set $J=(x_1y_2,x_2y_3,\ldots,x_{n-1}y_n)$, $I=J+(x_1y_n)$, and, for $1\leq k\leq m,$
 $I_k=I_{k-1}+(v_k),$ with $I_0=I.$ Therefore, $I_m=\ini_<(J_G)$.

\begin{Lemma}\label{lem1}
The ideals quotient $J:(x_1y_n)$ and $I_{k-1}:(v_k),$ for $k\geq 1,$ are minimally generated by regular sequences of monomials of length $n-1.$
\end{Lemma}

\begin{proof}
The statement is obvious for $J:(x_1y_n)$ since $J$ is minimally generated by a regular sequence. Now let $k\geq 1$ and let 
$v_k=x_ix_{j+1}\cdots x_ny_1\cdots y_{i-1}y_j$ for some $i,j$ with $2\leq j-i\leq n-2.$ 
Then $I_{k-1}:(v_k)=I:(v_k)+(v_1,\ldots,v_{k-1}):(v_k)$. One easily observes that $(v_1,\ldots,v_{k-1}):(v_k)=(x_{1},\ldots,x_{i-1},y_{j+1},\ldots,y_n).$ Hence, 
\[
I:(v_k)=(x_1,\ldots,x_{i-1},x_iy_{i+1},x_{i+1}y_{i+2},\ldots,x_{j-2}y_{j-1},x_{j-1}y_j,
y_{j+1},\ldots,y_n):(v_k)=
\]
\[
(x_1,\ldots,x_{i-1},y_{i+1},x_{i+1}y_{i+2},\ldots,x_{j-2}y_{j-1},x_{j-1},
y_{j+1},\ldots,y_n)
\]
\end{proof}

\begin{Remark}\label{degree}{\em 
From the above proof we also note that if $v_k=x_ix_{j+1}\cdots x_ny_1\cdots y_{i-1}y_j,$ then the regular sequence of monomials which generates $I_{k-1}:(v_k)$ contains $j-i-2$ monomials of degree $2$ and $n-j+i+1$ variables.}
\end{Remark}

In the following lemma we compute the projective dimension and the regularity of $S/I.$ This will be useful for the inductive study of the invariants of  $S/I_k.$

\begin{Lemma}\label{I0}
We have $\projdim S/I=n-1$ and $\reg S/I=n-2.$
\end{Lemma}

\begin{proof}
In the exact sequence 
\begin{equation}\label{seq1}
0\to \frac{S}{J:(x_1y_n)}(-2)\stackrel{x_1y_n}{\To}\frac{S}{J}\To \frac{S}{I}\to 0,
\end{equation}
we have $\projdim S/J:(x_1y_n) =\projdim S/J=n-1$ since both ideals are generated by regular sequences of length $n-1.$ Moreover, since $J$ is generated by a regular sequence of monomials of degree $2,$ by using the Koszul complex, we get $$\beta_{ij}(S/J)=\left\{
\begin{array}{ll}
	0, & j=2i,\\
	{n-1\choose i}, & j\neq 2i.
\end{array}
\right.$$ In particular, it follows that $\Tor_{n-1}(S/J,K)\cong K(-2n+2)$. 

Analogously, since $J:(x_1y_n)$ is generated by a regular sequence of $n-3$ monomials of degree $2$ and two variables, it 
follows that $\Tor_{n-1}(S/J:(x_1y_n),K)\cong K(-2n+4),$ which implies that $\Tor_{n-1}(S/J:(x_1y_n)(-2),K)\cong K(-2n+2).$ By the long exact sequence of Tor's derived from sequence~(\ref{seq1}), we get $\Tor_n(S/I,K)=0, $ hence $\projdim S/I\leq n-1.$

On the other hand, we have $$\beta_{n-2,j}(S/J:(x_1y_n))=\left\{
\begin{array}{ll}
	2, & j=2n-5,\\
	n-3, & j=2n-6,\\
	0,& \text{ otherwise}.
\end{array}
\right.$$ From the exact sequence of Tor's applied to (\ref{seq1}), as 
\[
\Tor_{n-1}(S/J,K)_{2n-3}=\Tor_{n-2}(S/J,K)_{2n-3}=0,
\]
we get the following exact sequence
\[
0\to \Tor_{n-1}(S/I,K)_{2n-3}\to \Tor_{n-2}(S/J:(x_1y_n),K)_{2n-5}\to 0.
\] Thus, $\Tor_{n-1}(S/I,K)\neq 0$ which implies that $\projdim S/I=n-1.$

For the regularity, we first observe that since the Koszul complex of the minimal generators gives the minimal graded free 
resolution of $S/J$ and, respectively, $S/J:(x_1y_n)$, we have $\reg S/J=\reg(S/J:(x_1y_n)(-2))=n-1.$ Then sequence~(\ref{seq1}) implies that $\reg S/I\leq n-1.$  We have observed above that $\Tor_{n-1}(S/I,K)_{2n-3}\neq 0,$ thus $\reg S/I\geq n-2.$ In order to derive the equality $\reg S/I= n-2$ we have to show that $\beta_{i,i+n-1}(S/I)=0$ for all $i.$

For $i=n-1$, as $\Tor_{n-2}(S/J:(x_1y_n),K)_{2n-4}=0,$ we get
\[
0\to \Tor_{n-1}(S/J:(x_1y_n),K)_{2n-4}\to \Tor_{n-1}(S/J,K)_{2n-2}\to \Tor_{n-1}(S/I,K)_{2n-2}\to 0.
\] But $\dim_K \Tor_{n-1}(S/J:(x_1y_n),K)_{2n-4}=\dim_K \Tor_{n-1}(S/J,K)_{2n-2}=1,$ which implies that  $\Tor_{n-1}(S/I,K)_{2n-2}=0.$ 
For $i<n-1,$  $\Tor_i(S/I,K)_{i+n-1}=0$ since $\Tor_i(S/J,K)_{i+n-1}=0$ and $\Tor_{i-1}(S/J:(x_1y_n),K)_{i+n-3}=0$. The latter equality holds since, as we have already observed,  $\reg S/J:(x_1y_n)=n-3.$
\end{proof}

\begin{Lemma}\label{induction}
For $1\leq k\leq m,$ we have $\projdim S/I_k\leq n$ and $\reg S/I_k\leq n-2.$
\end{Lemma}

\begin{proof}
We proceed by induction on $k,$ by using the following exact sequence and  Lemma~\ref{I0} for the initial step,
\begin{equation}\label{seq2}
0\to \frac{S}{I_{k-1}:(v_k)}(-\deg v_k)\to \frac{S}{I_{k-1}}\to \frac{S}{I_k} \to 0.
\end{equation}
Indeed, as $I_{k-1}:(v_k)$ is generated by a regular sequence of length $n-1$, it follows that  $\projdim S/I_{k-1}:(v_k)=n-1
.$ Thus, 
if $\projdim S/I_{k-1}\leq n,$ by (\ref{seq2}), it follows that $\projdim S/I_k\leq n.$

In addition, we have $\reg (S/I_{k-1}:(v_k))(-\deg v_k)=\reg (S/I_{k-1}:(v_k))+\deg v_k.$ If $v_k=x_ix_{j+1}\cdots x_ny_1
\cdots y_{i-1}y_j,$ by using Remark~\ref{degree}, we obtain $\reg (S/I_{k-1}:(v_k))=j-i-2.$ As $\deg v_k=n-j+i+1,$ we get
$\reg (S/I_{k-1}:(v_k))(-\deg v_k)=n-1.$ Let us assume, by induction, that $\reg S/I_{k-1}\leq n-2.$ Then, by using the  
sequence~(\ref{seq2}), we obtain 
\[
\reg S/I_k\leq \max\{\reg(S/I_{k-1}:(v_k))(-\deg v_k)-1,\reg S/I_{k-1} \}=n-2.
\]
\end{proof}

\begin{Proposition}\label{pd}
We have $\projdim S/\ini_<(J_G)=n$ and $\reg S/\ini_<(J_G)=n-2.$
\end{Proposition}

\begin{proof}
As $\ini_<(J_G)=I_m,$ by Lemma~\ref{induction}, we get $\projdim S/\ini_<(J_G)\leq n$ and $\reg S/\ini_<(J_G)\leq n-2.$ For the other inequalities, we use \cite[Theorem 3.3.4]{HH10}  which gives the inequalities $n=\projdim S/J_G\leq \projdim S/\ini_<(J_G)$ and $n-2=\reg S/J_G\leq \reg S/\ini_<(J_G).$ 
\end{proof}

\begin{Theorem}\label{final}
Let $G$ be a cycle. Then $S/\ini_<(J_G)$ and $S/J_G$ have the same extremal Betti number, namely 
$\beta_{n,2n-2}(S/J_G)=\beta_{n,2n-2}(S/\ini_<(J_G))={n-1\choose 2}-1.$
\end{Theorem}

\begin{proof}
We only need to show that $\beta_{n,2n-2}(S/\ini_<(J_G))=m$ since $m=(n^2-3n)/2={n-1\choose 2}-1$.

We use again the sequence~(\ref{seq2}). By considering its long exact sequence of Tor's and using the equality 
$\Tor_{n-1}(S/I_{k-1},K)_{2n-2}=0,$ we get
\[
0\to\Tor_n(S/I_{k-1},K)_{2n-2}\to \Tor_n(S/I_{k},K)_{2n-2}\to \Tor_{n-1}(S/I_{k-1}:(v_k),K)_{2n-2-\deg v_k}\to 0,
\]
for $1\leq k\leq m.$ By Remark~\ref{degree}, if $v_k=x_ix_{j+1}\cdots x_ny_1\cdots y_{i-1}y_j,$ then $I_{k-1}:(v_k)$ is generated by a regular sequence of monomials which contains $j-i-2$ elements of degree $2$ and $n-j+i+1$ variables. 
This implies that $\dim_K\Tor_{n-1}(S/I_{k-1}:(v_k),K)_{2n-2-\deg v_k}=1.$ Therefore, we get
\[
\dim_K\Tor_n(S/I_{k},K)_{2n-2}=\dim_K \Tor_n(S/I_{k-1},K)_{2n-2}+1
\] for $1\leq k\leq m.$ By summing up all these equalities, it follows that 
\[
\beta_{n,2n-2}(S/\ini_<(J_G))=\dim_K\Tor_n(S/I_{m},K)_{2n-2}=\dim_K\Tor_n(S/I,K)_{2n-2}+m=m.
\] The last equality is due to Lemma~\ref{I0}.
\end{proof}

\begin{Remark}{\em  There are examples of graphs whose edge ideal have several extremal Betti numbers. For instance, the graph $G$ displayed below has two extremal Betti numbers which are equal to the extremal Betti numbers of $\ini_<(J_G).$

\begin{figure}[hbt]
\begin{center}
\psset{unit=1.2cm}
\begin{pspicture}(1,-2)(7,2)
\psline(1,1)(7,1)
\psline(2,1)(4,0)
\psline(3,1)(4,0)
\psline(5,1)(4,0)
\psline(6,1)(4,0)
\psline(4,-1)(4,0)

\rput(1,1){$\bullet$}
\rput(2,1){$\bullet$}
\rput(3,1){$\bullet$}
\rput(4,1){$\bullet$}
\rput(5,1){$\bullet$}
\rput(6,1){$\bullet$}
\rput(7,1){$\bullet$}
\rput(4,0){$\bullet$}
\rput(4,-1){$\bullet$}

\end{pspicture}
\end{center}
\label{example}
\end{figure}
}
\end{Remark}

{}

\end{document}